\documentclass[12 pt,reqno]{amsart}
\usepackage{amsmath,amssymb,amsthm,hyperref,mathtools}
\usepackage{longtable,}
\usepackage{a4wide}
\usepackage{color}
\usepackage[shortlabels]{enumitem}
\usepackage{cleveref}
\usepackage{tikz}
\usetikzlibrary{decorations.pathreplacing}
\usetikzlibrary{shapes,snakes}

\newtheorem{thm}{Theorem}[section]
\newtheorem{cor}[thm]{Corollary}

\theoremstyle{remark}

\theoremstyle{definition}

\newtheorem{defn}[thm]{Definition}

\newcommand{\cay}[2]{\mathrm{Cay}(#1,#2)}		
\newcommand{\dih}[1]{\mathrm{Dih}(#1)}			
\newcommand{\aut}[1]{\mathrm{Aut}(#1)}			

\def\Z{\mathbb{Z}}

\DeclarePairedDelimiter\abs{\lvert}{\rvert}

\DeclarePairedDelimiter\ang{\langle}{\rangle}
\DeclarePairedDelimiter{\set}{\lbrace}{\rbrace}

\title{Cayley Graphs on Non-Isomorphic Groups}

\author{Joy Morris}
\author{Adrian Skelton}

\address{Department of Mathematics and Computer Science\\
	University of Lethbridge\\
	Lethbridge, AB T1K 3M4\\
	Canada}

\email{joy.morris@uleth.ca}
\email{adrian.skelton@uleth.ca}

\begin{document}
	
\maketitle

\begin{abstract}
	A number of authors have studied the question of when a graph can be represented
	as a Cayley graph on more than one nonisomorphic group. In this paper we give conditions for when a Cayley graph on an abelian group can be represented as a Cayley graph on a generalized dihedral group, and conditions for when the converse is true.
\end{abstract}

\section{Introduction}

A Cayley (di)graph $\cay{G}{S}$ is a (di)graph whose vertices are the elements of a group $G$ and whose edges are determined by the connection set $S\subset G$ by the following rule: there is an arc from $u$ to $v$ if and only if $v=su$ for some $s\in S$. Usually, the identity $e\in G$ is omitted from $S$ to ensure the (di)graph does not have a loop at every vertex. Furthermore, if $S=S^{-1}$, then for every arc from $u$ to $v$, there is also an arc from $v$ to $u$. In this case, we may replace all pairs of arcs with undirected edges, and the resulting structure is a graph.

In this paper, we will only be working with Cayley graphs, and use standard notation. That is, graphs are represented by $\Gamma$, the vertex set of a graph $\Gamma$ is represented by $V(\Gamma)$ and we write $u\sim v$ to denote that $u$ is connected to $v$ by an edge. Additionally, the proofs presented here make use of the fact that a graph $\Gamma$ is Cayley on a group $G$ if and only if the automorphism group of $\Gamma$ contains a regular subgroup isomorphic to $G$.

It is of interest to researchers when a Cayley graph on a group $G$ is also a Cayley graph on some non-isomorphic group $H$. A paper by Joseph \cite{joseph} gives necessary and sufficient conditions for a Cayley digraph of prime-squared order to be isomorphic to a Cayley digraph of both groups of the same order. This result was extended, giving necessary and sufficient conditions for a Cayley digraph of the cyclic group of order $p^k$, where $p>1$ is prime, to be isomorphic to a Cayley graph on some other group of order $p^k$. The case of odd primes was handled in \cite{morris}, and the case where $p=2$, when both groups are abelian, was solved in \cite{kovacs}. In both of these cases, graphs that were Cayley on both groups were all lexicographic (``wreath") products of smaller graphs. Another paper \cite{dobson} subsequently gave a group theoretic version of the result from \cite{joseph}.

In \cite{smolcic}, Morris and Smol\v{c}i\'{c} categorized two families of graphs that are each Cayley on an abelian group and a non-abelian group. It was shown that all Cayley graphs of cyclic groups of even order (known as circulant graphs) are also Cayley on a dihedral group of the corresponding order. The second family of graphs are those that are Cayley on generalized dihedral groups; with certain restrictions to the connection set, it was shown that these graphs are also Cayley on an abelian group that contains a direct factor of order $2$.  The first result presented here builds on both of these. It directly generalizes the first result by showing that all Cayley graphs on any abelian group of even order are also Cayley on generalized dihedral groups of the same order. This can also be thought of as looking at the second result from the opposite perspective: starting from a Cayley graph on an abelian group of even order, when is it also a Cayley graph on a generalized dihedral group? Our second result considers when a Cayley graph on a generalized dihedral group is also Cayley on some abelian group of even order that does not necessarily contain a direct factor of order $2$.

\section{Theorems and Proofs}

We begin with a formal definition of a generalized dihedral group:

\begin{defn}\label{dih}
	Let $A$ be an abelian group. Define the group $\dih{A}=\ang{A,x}$ where $x^2=1$ and $ax=xa^{-1}$ for every $a\in A$.
\end{defn}

Note that if $A$ is cyclic then this is the regular dihedral group, and if $A$ is an elementary abelian $2$-group, then $\dih{A}$ is isomorphic to the elementary abelian $2$-group whose rank is one greater than the rank of $A$.

In \cite{smolcic}, the following theorem was stated and proven although it is well-known in the field:

\begin{thm}[\cite{smolcic}]\label{1}
	Let $A$ be a cyclic group of even order, and let $D$ be the dihedral group of the same order. Let $S\subseteq A$ be closed under inversion, and let $\Gamma=\cay{A}{S}$. Then $\Gamma$ is also a Cayley graph on $D$. 
\end{thm}

Our first theorem generalises Theorem~\ref{1} from cyclic groups to any abelian group (still of even order). More precisely, we show that every Cayley graph on a group $C_{2^k}\times A$  where $k \ge 1$ is also a Cayley graph on $\dih{C_{2^{k-1}}\times A}$. Like Theorem~\ref{1}, this result is not necessarily true for digraphs, so we take $S$ to be closed under inversion.

\begin{thm}\label{2}
	Let $G$ be a finite abelian group of even order, so $G=\ang{c}\times A$ for some $c$ of order $2^k$, $k\ge1$, and some $A<G$. Let $S\subseteq G$ be closed under inversion and let $\Gamma=\cay{G}{S}$. Let $H=\ang{c^2}\times A$. Then $\Gamma$ is a Cayley graph on the generalized dihedral group $\dih{H}$.
\end{thm}

\begin{proof}
	Given any $h\in H$, we define the map $\alpha_h$ by $\alpha_h(z)=zh$ for each $z\in V(\Gamma)$, and define $\beta$ on $V(\Gamma)$ by $\beta(z)=z^{-1}g$ where $g=(c,e_A)\in G$. We claim that $\ang{\alpha_h,\beta\,:\,h\in H}\cong\dih{H}$ is a regular subgroup of $\aut{\Gamma}$.
	
	First we show that $\ang{\alpha_h,\beta}\cong\dih{H}$. It is easy to see that $\ang{\alpha_h}\cong H.$ Furthermore, since $G$ is abelian, we have that $$\beta^2(z)=\beta(z^{-1}g)=(z^{-1}g)^{-1}g=g^{-1}zg=z$$ so $\beta$ has order 2. Finally, again since $G$ is abelian, we have that $$\beta^{-1}\alpha_h\beta(z)=\beta\alpha_h(z^{-1}g)=\beta(z^{-1}gh)=(z^{-1}gh)^{-1}g=h^{-1}g^{-1}zg=zh^{-1}=\alpha_h^{-1}(z)$$ so $\beta$ inverts each $\alpha_h$. Thus $\ang{\alpha_h,\beta}\cong\dih{H}$ as desired.
	
	Next we show that $\ang{\alpha_h,\beta}$ is a regular subgroup of $\aut{\Gamma}$. Let $u\sim v$ be adjacent vertices of $\Gamma$, so that $su=v$ for some $s\in S$. Then for any $h\in H$ we have that 
	$$s\alpha_h(u)=suh=vh=\alpha_h(v).$$ 
	Since $s\in S$, this means that $\alpha_h(u)\sim\alpha_h(v)$, so $\alpha_h$ is an automorphism for each $h\in H$. Furthermore, since $G$ is abelian, we have that $$s^{-1}\beta(u)=s^{-1}u^{-1}g=(us)^{-1}g=v^{-1}g=\beta(v).$$ Since $s^{-1}\in S$, this means that $\beta(u)\sim\beta(v)$, so $\beta$ is an automorphism, and $\ang{\alpha_h,\beta}$ is a subgroup of $\aut{\Gamma}$. 
	
	In order to show that it acts regularly on $V(\Gamma)$, it suffices to show it acts transitively on $V(\Gamma)$, since $\abs{\ang{\alpha_h,\beta}}=\abs{G}$. Consider the arbitrary vertices $u=(c^i,a)$ and $v=(c^j,b)$. If $i$ and $j$ have the same parity, then $j=i+2m$ for some $m\in\Z$. Then since $G$ is abelian, and with $h=(c^{2m},a^{-1}b)$, we have that $$v=(c^j,b)=(c^{i+2m},b)=(c^i,a)(c^{2m},a^{-1}b)=uh=\alpha_h(u).$$
	If $i$ and $j$ have different parity then $1-i$ and $j$ have the same parity. Then by the previous argument, there exists some $h\in H$ such that
	$$v=(c^j,b)=\alpha_h((c^{1-i},a^{-1}))=\alpha_h((c^i,a)^{-1}(c,e))=\alpha_h(u^{-1}g)=\alpha_h\beta(u).$$ Thus $\ang{\alpha_h,\beta}\cong \dih{H}$ is a group of automorphisms that acts regularly on the vertices of $\Gamma$ and so $\Gamma$ is a Cayley graph on $\dih{H}$.
\end{proof}

The requirement that the abelian group we start with has even order is essential.  There is no analogous result even for groups of order $p^k$ where $p$ is an odd prime. This was confirmed via computer, where a Cayley graph on $C_9\times C_3$ was found whose automorphism group of contains only one regular subgroup. The connection set for this specific graph was $S=\set{(a,e_{C_3}), (e_{C_9},b), (a^{-1},e_{C_3}), (e_{C_9},b^{-1})}$, where $a$ generates $C_9$ and $b$ generates $C_3$.

In the context of Theorem~\ref{2}, if we consider abelian groups of even order that have multiple cyclic groups whose orders are distinct power of 2 as direct factors,
we notice that the result can be applied on each of these direct factors. In this situation, the theorem tells us immediately that a Cayley graph on the abelian group is also Cayley on more than one generalized dihedral group.

\begin{cor}\label{cor}
	Let $A$ be any abelian group of odd order, and let $\Gamma$ be a Cayley graph on the abelian group $A\times (C_{2^{n_1}})^{m_1}\times\cdots\times (C_{2^{n_k}})^{m_k}$. Then $\Gamma$ is also Cayley on $$\dih{A\times C_{2^{n_1-1}}\times(C_{2^{n_1}})^{m_1-1}\times\cdots\times (C_{2^{n_k}})^{m_k}}, \ldots,$$ $$\dih{A\times (C_{2^{n_1}})^{m_1}\times\cdots\times C_{2^{n_k-1}}\times(C_{2^{n_k}})^{m_k-1}}.$$
	
	In general,	if a group $G$ can be written as a direct product of an abelian group of odd order and multiple cyclic groups of order some power of $2$, then the Cayley graph of $G$ is also a Cayley graph on $k$ generalized dihedral groups, where $k$ is the number of distinct powers of $2$ that show up in the orders of the cyclic groups.
\end{cor}

For an example of Corollary~\ref{cor}, if $\Gamma$ is a Cayley graph on the group $C_4\times C_2$ it is also Cayley on the group $\dih{C_2\times C_2}=C_2\times C_2\times C_2$ as well as the group $\dih{C_4}=D_4$.
In general, if $\Gamma$ is Cayley on $C_4\times (C_2)^k$, then it is also Cayley on the elementary abelian $2$-group $(C_2)^{k+2}$.

For a bigger example of Corollary~\ref{cor} in action, take $\Gamma$ to be a Cayley graph on the group $C_8\times C_4\times C_2\times A$ where $A$ is abelian. Then $\Gamma$ is also Cayley on $\dih{C_4\times C_4\times C_2\times A}$, $\dih{C_8\times C_2\times C_2\times A}$ and $\dih{C_8\times C_4\times A}$.

The other main theorem in~\cite{smolcic} states:
\begin{thm}[\cite{smolcic}]\label{JS}
	Let $A$ be an abelian group, and let $D=\dih{A}$ be the corresponding generalized dihedral group. Let $S\subseteq D$ be closed under inversion, and let $\Gamma=\cay{D}{S}$. 
	
	Suppose there is some $y\in xA$ such that for every $a\in A$ we have $ya\in S\cap xA$ if and only if $ya^{-1}\in S\cap xA$. Then $\Gamma$ is also a Cayley graph on the abelian group $A\times C_2$. 
\end{thm}

Our next theorem is a generalization of  Theorem~\ref{JS}. It gives a condition for when a Cayley graph on a generalized dihedral group is also Cayley on an abelian group, though rather than moving from $\dih{A}$ to $C_2\times A$, here we move from $\dih{C_{2^k}\times A}$ to $C_{2^{k+1}}\times A$. Taking the special case $k=0$ in Theorem~\ref{3} gives Theorem~\ref{JS}.

\begin{thm}\label{3}
	Let $G$ be a finite abelian group and let $D=\dih{G}$. Let $S \subseteq D$ be closed under inversion, and $\Gamma=\cay{D}{S}$. 
	Suppose that there exists some $c\in G$ such that $c^{2^k}=e$ and $G$ can be written as the internal direct product $G=\ang{c}\times A$. 
	
	If there exists some $y \in xG$ such that for every $g \in G$ we have $yg\in S\cap xG$ if and only if $yg^{-1}c\in S\cap xG$, then $\Gamma$ is a Cayley graph on the abelian group $C_{2^{k+1}}\times A$.
\end{thm}

\begin{proof}
	For every $g\in G$, define the map $\alpha_g$ given by $\alpha_g(z)=zg$ for all $z\in V(\Gamma)=D$. Let $y$ be chosen to satisfy our assumptions, and define the map $\beta$ by $\beta(z)=yz$ if $z\in G$, and $\beta(z)=ycz$ if $z\in xG$. Let $H=\ang{\alpha_a,\beta\,:\,a\in A}.$ We claim that $H\cong C_{2^{k+1}}\times A$ is a regular subgroup of $\aut{\Gamma}$.
	
	First we show $H\cong C_{2^{k+1}}\times A$. It should be clear that $\ang{\alpha_a \,:\,a\in A}\cong A$. If $z\in G$ then since $G$ is abelian we have $$\beta^2(z)=\beta(yz)=ycyz=c^{-1}z=zc^{-1}=\alpha_{c^{-1}}(z).$$ 
	Furthermore, if $z\in xG$, then 
	$$\beta^2(z)=\beta(ycz)=yycz=cz=zc=\alpha_{c}(z).$$
	Since both $c$ and $c^{-1}$ have order $2^k$ as do $\alpha_{c}$ and $\alpha_{c^{-1}}$, it is straightforward to observe that $\beta^2$ has order $2^k$, and therefore $\beta$ has order $2^{k+1}$. Since $A\subset G$ is abelian, to show that $H$ is abelian, it suffices to show that $\beta$ commutes with $\alpha_a$ for each $a\in A$. If $z\in G$ then $za\in G$, and we have that
	$$\alpha_a\beta(z)=\alpha_a(yz)=yza=\beta(za)=\beta\alpha_a(z).$$
	If $z\in xG$ then $za\in xG$ and we have that
	$$\alpha_a\beta(z)=\alpha_a(ycz)=ycza=\beta(za)=\beta\alpha_a(z).$$
	So, $H\cong C_{2^{k+1}}\times A$ as desired.
	
	Next we show that $H$ is a subgroup of $\aut{\Gamma}$. Let $u\sim v$ be adjacent vertices, so that $su=v$ for some $s\in S$. It is easy to see that for any $a\in G$ we have $$s\alpha_a(u)=sua=va=\alpha_a(v),$$ 
	so $\alpha_a(u)\sim\alpha_a(v)$, thus each $\alpha_a$ is an automorphism of $\Gamma$. 
	
	Due to our definition of $\beta$, we will need to consider four cases depending on which coset $s$ and $u$ belong to. We will use the fact that $S$ is closed under inverses, and the fact that $g^{-1}y=yg$ for every $g\in G$, which follows immediately from the definition of $x$ in Definition~\ref{dih}, and the fact that $y \in xG$. For our first two cases, we take $s\in G$. If $u\in G$ then since $su=v$, we have that $v\in G$ as well. It then follows that
	$$s^{-1}\beta(u)=s^{-1}yu=ysu=yv=\beta(v).$$
	This means that $\beta(u)\sim\beta(v)$. Next, if $u\in xG$, then $v\in xG$. Since $G$ is abelian, we have that 
	$$s^{-1}\beta(u)=s^{-1}ycu=yscu=ycsu=ycv=\beta(v),$$
	so $\beta(u)\sim\beta(v)$ again. 
	
	For the third and fourth cases, we take $s\in xG$, so we can write $s=yg$ for some $g \in G$. Then by assumption $yg^{-1}c$ is also in $S$. If $u\in xG$, then $v\in G$. Since $G$ is abelian, we have that
	$$yg^{-1}c\beta(u)=yg^{-1}cycu=y(yc^{-1}g)cu=ysu=yv=\beta(v),$$
	so again we have that $\beta(u)\sim\beta(v)$. Finally, if $u\in G$, then $v\in xG$. Once again since $G$ is abelian we have that 
	$$yg^{-1}c\beta(u)=yg^{-1}cyu=yc(g^{-1}y)u=yc(yg)u=ycsu=ycv=\beta(v),$$
	so in all four cases, if $u\sim v$ then $\beta(u)\sim\beta(v)$, thus $\beta$ is an automorphism of $\Gamma$.
	
	To show that it is a regular subgroup, it suffices to show that it acts transitively on the vertices of $\Gamma$, since $\abs{H}=\abs{V(\Gamma)}$. Since acting on a fixed coset of $G$ we either have $\beta^2=\alpha_{c^{-1}}$ or $\beta^2=\alpha_c$, it should be clear that on either coset
	$$\ang{\alpha_a, \beta^2 \,:\,a\in A}\cong \ang{\alpha_g\,:\,g\in G}\cong G,$$ so $H$ is transitive on each coset of $G$. Finally, if $z\in G$ then $\beta(z)=xz\in xG$, and if $z\in xG$, then $\beta(z)=xcg\in G$, so $\beta$ interchanges the cosets of $G$. Thus $H$ is transitive, and thus regular, on $D=V(\Gamma)$.	
\end{proof}

We have shown that Cayley graphs on abelian groups of even order are also Cayley graphs on one or more corresponding dihedral groups, and we have given conditions for when a Cayley graph on $\dih{C_{2^k}\times A}$ is also a Cayley graph on $C_{2^{k+1}}\times A$.
It is unknown if the restriction on the connection set in Theorem~\ref{3} is necessary.
It is possible that the conditions given here exclude some graphs that are Cayley on both groups and that there exist less restrictive conditions that exclude less graphs.

An interesting observation came up in computational examination of regular subgroups of automorphism groups of Cayley graphs, that could be used as inspiration for future projects. A Cayley digraph was found that could be represented on both the quaternion group and on $C_8$. This is an example of the result from \cite{kovacs} not holding unless both groups are abelian. It would be interesting to study whether or not we can find conditions for Cayley digraphs of cyclic groups of the proper order to be Cayley digraphs on generalized dicyclic groups.


\begin{thebibliography}{99}
	\bibitem{dobson} E.~Dobson and D.~Witte, \textit{Transitive permutation groups of prime-squared degree}, J. Algebraic Combin. \textbf{16} (2002), no. 1, 43--69.
	
	\bibitem{joseph} A.~Joseph, \textit{The isomorphism problem for Cayley digraphs on groups of prime-squared order}, Discrete Math. \textbf{141} (1995), no. 1-3, 173--183.
	
	\bibitem{kovacs} I.~Kov\'{a}cs and M.~Servatius, \textit{On Cayley digraphs on non-isomorphic 2-groups}, J. Graph Theory \textbf{70} (2012), no. 4, 435--448.
	
	\bibitem{morris} J.~Morris, \textit{Isomorphic Cayley graphs on non-isomorphic groups}, J. Graph Theory	\textbf{31} (1999), no. 4, 345--362.
	
	\bibitem{smolcic} J.~Morris and J.~Smol\v{c}i\'{c}, \textit{Two families of graphs that are Cayley on nonisomorphic groups}, J. Algebra Combinatorics Discrete Structures and Applications \textbf{8} (2021), no. 1, 53--57.
\end{thebibliography}
\end{document}